\documentclass[reqno]{amsart}

\numberwithin{equation}{section} 
\usepackage{amsmath,amsfonts,amssymb,amsthm,latexsym}
\usepackage{enumerate}
\usepackage{cancel}
\usepackage{stackrel}
\usepackage{empheq}

\usepackage{cases}
\usepackage[square,comma,numbers,sort&compress]{natbib}

\usepackage[colorlinks=true]{hyperref}
\hypersetup{urlcolor=blue, citecolor=red}

\newcounter{mnote}
\theoremstyle{plain}
\newtheorem{theorem}{Theorem}[section]
\newtheorem{proposition}[theorem]{Proposition}
\newtheorem{lemma}[theorem]{Lemma}

\newtheorem{definition}[theorem]{Definition}

\theoremstyle{remark}
\newtheorem{remark}[theorem]{Remark}

\newcommand{\vect}[1]{\mathbf{#1}}

\newcommand{\bu}{\vect{u}}
\newcommand{\bv}{\vect{v}}

\newcommand{\bx}{\vect{x}}

\newcommand{\field}[1]{\mathbb{#1}}
\newcommand{\nN}{\field{N}}
\newcommand{\nT}{\field{T}}
\newcommand{\nZ}{\field{Z}}

\newcommand{\nR}{\field{R}}

\newcommand{\vphi}{\varphi}
\newcommand{\veps}{\varepsilon}

\newcommand{\sand}{\quad\text{and}\quad}

\newcommand{\pd}[2]{\frac{\partial #1}{\partial #2}}
\newcommand{\od}[2]{\frac{d #1}{d #2}}

\newcommand{\abs}[1]{\left\lvert#1\right\rvert}
\newcommand{\norm}[1]{\left\lVert#1\right\rVert}
\newcommand{\set}[1]{\left\{#1\right\}}

\newcommand{\strong}{\rightarrow}

\newcommand{\weakstar}{\stackrel[\ast]{}{\rightharpoonup}}
\newcommand{\convh}{\stackrel[h]{}{\ast}}

\newcommand{\LpP}[1]{\text{$L^{#1}(\nT^3)$}}

\newcommand{\LphP}[1]{\text{$L^{#1}_h(\nT^2)$}}

\newcommand{\WP}[2]{\text{$W^{#1,#2}(\nT^3)$}}
\newcommand{\WhP}[2]{\text{$W^{#1,#2}_h(\nT^2)$}}
\newcommand{\HhP}[1]{\text{$H^{#1}_h(\nT^2)$}}
\newcommand{\HP}[1]{\text{$H^{#1}(\nT^3)$}}

\newcommand{\LtLphP}[2]{\text{$L^{#1}([0,T];L^{#2}_h(\nT^2))$}}

\begin{document}
\title[Global Regularity for Slow Limiting Ocean Dynamics Model]{Global Regularity for an Inviscid Three-dimensional Slow Limiting Ocean Dynamics Model}

\date{November 24, 2013}

\author{Chongsheng Cao}
\address[Chongsheng Cao]{Department of Mathematics\\
                Florida International University\\
                Miami, FL 33199, USA}
\email[Chongsheng Cao]{caoc@fiu.edu}
\author{Aseel Farhat}
\address[Aseel Farhat]{Department of Mathematics\\
                Indiana University, Bloomington\\
        Bloomingtin, IN 47405, USA}
\email[Aseel Farhat]{afarhat@indiana.edu}
\author{Edriss S. Titi}
\address[Edriss S. Titi]{Department of Mathematics, and Department of Mechanical and Aero-space Engineering\\
University of California\\
Irvine, CA 92697, USA.
Also, The Department of Computer Science and Applied Mathematics\\
The Weizmann Institute of Science, Rehovot 76100, Israel.}
\email[Edriss S. Titi]{etiti@math.uci.edu and edriss.titi@weizmann.ac.il}

\begin{abstract}
We establish, for smooth enough initial data, the global well-posedness (existence, uniqueness and continuous dependence on initial data) of solutions, for an inviscid three-dimensional {\it slow limiting ocean dynamics} model. This model was derived as a strong rotation limit of the rotating and stratified Boussinesg equations with periodic boundary conditions. To establish our results we utilize the  tools developed  for investigating the two-dimensional incompressible Euler equations and linear transport equations. Using a weaker formulation of the model we also show the global existence and uniqueness of solutions, for less regular initial data.

\end{abstract}
\maketitle
\thispagestyle{empty}

\centerline{\it Dedicated to Professor Marshall Slemrod on the occasion of his $70^{th}$ birthday.}

\section{Introduction}

The questions of global well-posedness of the three-dimesional Navier-Stokes and Euler equations, as well as the three-dimensional Boussinesq equations of incompressible flows, are considered to be among the most challenging mathematical problems in applied analysis. In the context of the atmosphere and the ocean circulation dynamics, geophysicists take advantage of the fast rotation (small Rossby number $Ro$) effect to simplify the Boussinesq equations. The Taylor-Proudman theorem \cite{Pedlosky_1987} suggests that the fluid velocity will be uniform along any line parallel to the axis of rotation and that the fluid motion takes place in tall columnar structures. In a recent work \cite{Wingate_Embid_Cerfon_Taylor_2011}, the authors explored the fast rotation limit (Rossby number $Ro \rightarrow 0$) of the rotating stratified Boussinesq equations with periodic boundary conditions to derive a system for the ``slow" dynamics. Subject to periodic boundary conditions in $\nT^3 = [0,L]^3$, the viscous version of the {\it slow limiting dynamics model,} that was introduced in \cite{Wingate_Embid_Cerfon_Taylor_2011}, is given by:
\begin{subequations}\label{slow_dynamics_eq}
\begin{align}
\frac{\partial \bu_h}{\partial t} + (\bu_h\cdot\nabla_h)\bu_h+ \nabla_h p &= \frac{1}{Re} \Delta_h\bu_h, \label{slow_dynamics_eq_a} \\
\frac{\partial w}{\partial t} + (\bu_h\cdot\nabla_h)w &=  \frac{1}{Re} \Delta_hw - \frac{1}{Fr}\left<\rho\right>_z, \label{slow_dynamics_eq_b}
\end{align}
\begin{align}
\frac{\partial \rho}{\partial t} + (\bu\cdot\nabla)\rho - \frac{1}{Fr}w &= \frac{1}{Re Pr} \Delta\rho, \label{slow_dynamics_eq_c}\\
\nabla_h\cdot \bu_h = 0, \qquad \nabla\cdot\bu &= 0, \label{slow_dynamics_eq_d}
\end{align}
\end{subequations}
where $\bu = (\bu_h, w) = (\bu_h(t;x,y), w(t;x,y))$ is the velocity vector field, $p = p(t;x,y)$ is the pressure, $\rho = \rho(t;x,y,z)$ is the density fluctuation, $\nabla_h = (\frac{\partial}{\partial x}, \frac{\partial}{\partial y})$, $\Delta_h = \frac{\partial^2}{\partial x^2} + \frac{\partial^2}{\partial y^2}$, $\Delta = \frac{\partial^2}{\partial x^2} + \frac{\partial^2}{\partial y^2}+ \frac{\partial^2}{\partial z^2}$, $Re$ is the Reynolds number, $Fr$ is the Froude number and $Pr$ is the Prandtl number. Also, $\left<\rho\right>_z$ is the density average in the vertical direction defined by
$$ \left<\rho\right>_z(t;x,y) := \frac{1}{L} \int_{[0,L]} \rho(t;x,y,z) \, dz. $$

The derivation of this model is based on the assumption that the solution evolves only on the slow advective time scale. For this reason the system was called a {\it slow limiting dynamics model,} even though it was derived in the limit of fast rotation. If the initial data data contains inertial waves then the model has to be modified to take into account the fast inertial waves.


In the {\it slow limiting dynamics model} \eqref{slow_dynamics_eq}, the horizontal component of the velocity $\bu_h$ is governed by the 2D Navier-Stokes equations. Moreover, $\bu_h$ evolves independently of the vertical velocity $w$ and the density fluctuation $\rho$, but it influences the dynamics of these variables through the advection terms in \eqref{slow_dynamics_eq_b} and \eqref{slow_dynamics_eq_c}. The dynamics of the vertical velocity $w$ and the density fluctuation $\rho$ are strongly coupled. Interestingly, the vertical velocity $w$ evolves according to a two-dimensional forced advection-diffusion equation, \eqref{slow_dynamics_eq_b}, with buoyancy force given by $\left<\rho\right>_z$, the density average in the vertical direction. However, the evolution equation of the density $\rho$ in \eqref{slow_dynamics_eq_c} retains its three-dimensionality.

The authors in \cite{Wingate_Embid_Cerfon_Taylor_2011} performed forced numerical simulations of the rotating Boussinesq equations to demonstrate a support for the theory in the limit as $Ro \rightarrow 0$. They find the formulation and the presence of large-scale columnar Taylor-Proudman flows, as well, they show that the ratio of the ``slow" total energy to the total energy approaches to a constant; and that at very small Rossby numbers $Ro$ this constant approaches to the value 1.

We notice that when we take the $z$-average of \eqref{slow_dynamics_eq_c} we conclude that $\left<\rho\right>_z$ satisfies the evolution equation:
\begin{align}
\pd{\left<\rho\right>_z}{t} + (\bu_h\cdot\nabla_h) \left<\rho\right>_z = \frac{1}{Fr}w + \frac{1}{Re Pr} \Delta_h\left<\rho\right>_z.
\end{align}
We introduce here the inviscid version of system \eqref{slow_dynamics_eq}:
\begin{subequations}
\label{inviscid_slow_dynamics_eq}
\begin{align}
&\pd{\bu_h}{t} + (\bu_h\cdot\nabla_h)\bu_h +\nabla_hp = 0,   \label{inviscid_slow_dynamics_eq_a}\\
&\pd{w}{t} + (\bu_h\cdot\nabla_h)w = -\frac{1}{Fr} \left<\rho\right>_z; \quad \left<\rho\right>_z(t;\bx_h) := \frac{1}{L}\int_0^L \rho(t;\bx)\,dz,\label{inviscid_slow_dynamics_eq_b}\\
&\pd{\left<\rho\right>_z}{t} + (\bu_h\cdot\nabla_h)\left<\rho\right>_z = \frac{1}{Fr}w , \label{viscous_slow_dynamics_eq_rho_z}\\
&\pd{\rho}{t}+ (\bu_h\cdot\nabla_h)\rho + w \pd{\rho}{z}= \frac{1}{Fr}w, \label{inviscid_slow_dynamics_eq_c}\\
&\nabla_h\cdot \bu_h = 0, \quad \pd{w}{z}  =0, \label{div_free_slow_dynamics_eq_vorticity}\\
&\bu_h(0;\bx_h) = \bu_h^0(\bx_h), \quad w(0;\bx_h) = w^0(\bx_h), \quad \rho(0,\bx)= \rho^0(\bx).
\end{align}
\end{subequations}

Denote by $\nT^d$ the $L$-period box $[0,L]^d$. In this work we will establish the global well-posedness of strong solutions, for smooth enough initial data, and the global existence and uniqueness of weak solutions, for less regular initial data, for the inviscid system \eqref{inviscid_slow_dynamics_eq} in the three-dimensional torus $\nT^3$, i.e. subject to periodic boundary conditions. 
This paper is organized as follows. In section \ref{Pre}, we recall the global well-posedness result of solutions for the two-dimensional incompressible Euler equations and the global existence and uniqueness result for linear transport equations. In section \ref{inviscid_well-posedness}, we prove a global existence and uniqueness of weak solutions result of \eqref{inviscid_slow_dynamics_eq}. Moreover, we prove a well-posedness (continuous dependence on initial data) result of strong solutions of system \eqref{inviscid_slow_dynamics_eq}.  In section \ref{inviscid_slow_dynamics_eq_vorticity} we introduce a presentation of system \eqref{inviscid_slow_dynamics_eq} in vorticity formulation and prove a global existence and uniqueness of weak solutions (without continuous dependence on initial data) result for the system in this presentation.

\bigskip
\section{Preliminaries}\label{Pre}
In this section, we introduce some preliminary material and notations which are commonly used in the mathematical study of fluids, in particular in the study of the Navier-Stokes equations (NSE) and the Euler equations.

Let $\mathcal{F}_h$, $\mathcal{F}$ be the set of all trigonometric polynomials of zero-average with periodic domain $\nT^2$ and $\nT^3$, respectively.  We define the spaces of smooth functions which incorporates the divergence-free and zero-average condition to be:
\[\mathcal{V}_h:=\set{\vphi\in\mathcal{F}_h:\nabla_h\cdot\vphi=0 \text{ and} \int_{\nT^2}\vphi\;dx=0},\]
\[\mathcal{V}:=\set{\phi\in\mathcal{F}:\nabla\cdot\phi=0 \text{ and} \int_{\nT^3}\phi\;dx=0}.\]
We denote by $\LphP{p}$, $\WhP{s}{p}$, ${H}_h^s(\nT^3)\equiv \WhP{s}{2}$ to be the closures of $\mathcal{V}_h$ in the usual Lebesgue and Sobolev spaces. Similarly, we denote by $\LpP{p}$, $\WP{s}{p}$, ${H}^s(\nT^3)\equiv \WP{s}{2}$ to be the closures of $\mathcal{V}$ in the usual Lebesgue and Sobolev spaces, respectively.

Since we restrict ourselves to finding solutions over the three-dimensional $L$-periodic box $\nT^3$, therefore, if we assume that $ \int_{\nT^2} \bu_h^0 \, d\bx_h$ $={\mathbf 0}$, $\int_{\nT^2} w^0 \, d\bx_h$ $=$ $\int_{\nT^3}\rho^0\, d\bx$ $=0$, then integrating system \eqref{inviscid_slow_dynamics_eq} implies that
\begin{align*}
\int_{\nT^2} \bu_h(t;\bx_h)\, d\bx_h &= {\mathbf 0},\\
\od{}{t}\int_{\nT^2}w(t;\bx_h)\,d\bx_h &= -\frac{1}{Fr}\int_{\nT^2} \left<\rho\right>_z(t;\bx_h)\, d\bx_h, \\
\od{}{t}\int_{\nT^2} \left<\rho\right>_z(t;\bx_h)\, d\bx_h &= \frac{1}{Fr}\int_{\nT^2}w(t;\bx_h)\,d\bx_h,\\
\od{}{t}\int_{\nT^3}\rho(t;\bx)\,d\bx &= \frac{L}{Fr}\int_{\nT^2} w(t;\bx_h)\,d\bx_h,
\end{align*}
for any $t>0$. This yield that
\begin{align*}
\od{}{t}\left(\int_{\nT^2}w(t;\bx_h)\,d\bx_h + \int_{\nT^2} \left<\rho\right>_z(t;\bx_h)\, d\bx_h\right) &= \frac{1}{Fr}\left(\int_{\nT^2}w(t;\bx_h)\,d\bx_h - \int_{\nT^2} \left<\rho\right>_z(t;\bx_h)\, d\bx_h\right), \\
\od{}{t}\left(\int_{\nT^2}w(t;\bx_h)\,d\bx_h - \int_{\nT^2} \left<\rho\right>_z(t;\bx_h)\, d\bx_h\right) &= -\frac{1}{Fr}\left(\int_{\nT^2}w(t;\bx_h)\,d\bx_h + \int_{\nT^2} \left<\rho\right>_z(t;\bx_h)\, d\bx_h\right),
\end{align*}
for any $t>0$, and so $\int_{\nT^2}w(t;\bx_h)\,d\bx_h + \int_{\nT^2} \left<\rho\right>_z(t;\bx_h)\, d\bx_h$ $=$ $\int_{\nT^2}w(t,\bx_h)\,d\bx_h -\int_{\nT^2} \left<\rho\right>_z(t;\bx_h)\, d\bx_h$ $= 0$, for any $t>0$.
This implies that $ \int_{\nT^2}w(t;\bx_h)\,d\bx_h$ $=$ $ \int_{\nT^2} \left<\rho\right>_z(t;\bx_h)\, d\bx_h$ $=0$, for any $t\geq0$. This yields that
\begin{align*}
\od{}{t} \int_{\nT^3} \rho(t;\bx)\, d\bx =0,
\end{align*}
for any $t>0$. Thus, we can work in the spaces defined above consistently.

We define the inner products on $\LphP{2}$ and $\HhP{1}$, respectively, by
\[(\bu,\bv)_h=\sum_{i=1}^2\int_{\nT^2} u^iv^i\,d\bx_h
\sand
((\bu,\bv))_h=\sum_{i,j=1}^2\int_{\nT^2}\partial_ju^i\partial_jv^i\,d\bx_h,
\]
and the associated norms $\norm{\bu}_{\LphP{2}}=(\bu,\bu)_h^{1/2}$ and $\norm{\bu}_{\HhP{1}}=((\bu,\bu))_h^{1/2}$. Similarly,  we define the inner products on $\LpP{2}$ and $\HP{1}$ respectively by
\[(\bu,\bv)=\sum_{i=1}^3\int_{\nT^3} u^iv^i\,d\bx
\sand
((\bu,\bv))=\sum_{i,j=1}^3\int_{\nT^3}\partial_ju^i\partial_jv^i\,d\bx,
\]
and the associated norms $\norm{\bu}_{\LpP{2}}=(\bu,\bu)^{1/2}$ and $\norm{\bu}_{\HP{1}}=((\bu,\bu))^{1/2}$. (We use these notations indiscriminately for both scalars and vectors, which should not be a source of confusion). Note that $((\cdot,\cdot))_h$ and $((\cdot,\cdot))$ are norms due to the Poincar\'e inequality, Lemma \ref{Poincare}, below.

Let $Y$ be a Banach space.  We denote by $L^p([0,T];Y)$ the space of (Bochner) measurable functions $t\mapsto w(t)$, where $w(t)\in Y$, for a.e. $t\in[0,T]$, such that the integral $\int_0^T\|w(t)\|_Y^p\,dt$ is finite.

\begin{remark}
In this paper, $C$ represents a dimensionless constant that may change from line to line.
\end{remark}
We recall the well-known two-dimensional elliptic estimate, due to the Biot-Savart law, for $\nabla_h\cdot \bu_h=0$ and $\nabla_h\times\bu_h = \omega$,
\begin{equation}\label{elliptic_regularity_Yudovich}
\norm{\bu_h}_{W^{1,p}_h(\nT^2)} \leq C p \|\omega\|_{L^p_h(\nT^2)},
\end{equation}
for every $p\in [2,\infty)$ (see, e.g., \cite{Yudovich_1963} and references therein), where $C$ is a dimensionless constant, which is independent of $p$.


Furthermore, we have the Poincar\'e inequality:
\begin{lemma}\cite{Constantin_Foias_1988}\label{Poincare}
For all $\vphi\in \HhP{1}$ and $\phi \in \HP{1}$, we have
\begin{align}\label{poincare_h}
\|\vphi\|_{\LphP{2}}\leq C L\|\nabla_h\vphi\|_{\LphP{2}},
\end{align}
and
\begin{align}\label{poincare}
\|\phi\|_{\LpP{2}}\leq CL  \|\nabla\vphi\|_{\LpP{2}}.
\end{align}
\end{lemma}

Next, we recall the global existence and uniqueness theorem, due to Yudovich, \cite{Yudovich_1963}, for the incompressible two-dimensional Euler equations in vorticity formulations (see also \cite{Bardos_1972, Kato_1990, Majda_Bertozzi_2002}).

The two-dimensional Euler equations, for incompressible inviscid flows in the periodic box $\nT^2$ are
\begin{subequations}\label{2D_Euler}
\begin{align}
&\pd{\bu_h}{t} + (\bu_h\cdot\nabla_h) \bu_h + \nabla_h p = 0, \qquad  \text{in } [0,T]\times\nT^2\\
&\nabla_h\cdot \bu_h = 0, \qquad\qquad \qquad \qquad \qquad\; \text{in } [0,T]\times\nT^2 \end{align}\begin{align}
&\bu_h(0;\bx_h) = \bu_h^0(\bx_h), \qquad \qquad \qquad \;\;\; \text{in } \nT^2, 
\end{align}
\end{subequations}
where $T>0$ is given. Here, $\bu_h = \bu_h(t;x,y)$ is the velocity vector field, $p= p(t;x,y)$ is the pressure. The vorticity formulation, for the two-dimensional incompressible Euler equations is
\begin{subequations}\label{2D_Euler_vorticity}
\begin{align}
&\pd{\omega}{t} + (\bu_h\cdot\nabla_h) \omega  = 0,  \qquad \qquad \qquad \text{in } [0,T]\times\nT^2\\
&\nabla_h\cdot \bu_h = 0, \quad \omega = \nabla_h\times\bu_h, \qquad \;\;\;\text{in } [0,T]\times\nT^2 \\
&\omega(0;\bx_h) = \omega^0(\bx_h), \qquad \qquad \qquad \quad\; \text{in } \nT^2.
\end{align}
\end{subequations}
The velocity is determined from the vorticity by means of the two-dimensional periodic Biot-Savart law:
\begin{subequations}\label{conv}
\begin{align}
\bu_h(x,y) = K\convh\omega &:=\int_{\nT^2} K(x-s,y-\xi)\,\omega(s,\xi)\,dsd\xi,\\
K(x,y) &= \nabla_h^\perp G(x,y),
\end{align}
\end{subequations}
where $G(x,y)$ is the fundamental solution of the Poisson equation in two-dimensions subject to periodic boundary conditions, the binary operation $\convh$ denotes the horizontal convolution, and $\nabla_h^\perp = (-\pd{}{y}, \pd{}{x})$.

The questions of global well-posedness and the blowup of smooth solutions of the three-dimensional Euler equations has been studied by many authors. The {\it Beale--Kato--Majda} criterion \cite{Beale_Kato_Majda_1984} states that the quantity
\begin{align*}
\int_0^T \norm{\omega(t)}_{L^\infty} \, dt
\end{align*}
controls the blowup; that is if it is finite then the solution of the Euler equations remains as smooth as the initial data, for initial data $\omega_0\in H^s$, for $s>1$ in 2D and $s>3/2$ in 3D, on the time interval $[0,T]$, otherwise there is a finite blowup. For initial data $\bu_0\in H^s$, for $s>5/2$, the three-dimensional Euler equations posses a unique local in time solution $\bu(t;\bx)$ in the same space $H^s$ (cf. \cite{Beale_Kato_Majda_1984}, \cite{Majda_Bertozzi_2002}). The same result is valid for initial data $\bu_0 \in C^{1,\alpha}$ for $\alpha \in (0,1]$ \cite{Lichtenstein_1925}. The loss of smoothness of weak solutions for the three-dimensional Euler equations with initial data $\bu_0 \in C^{0,\alpha}$, with $\alpha \in (0,1)$ is shown in \cite{Bardos_Titi_2010}. In other words, the space $C^1$ is the critical space for the short time well-posedness of the three-dimensional Euler equations; that is for initial data more regular than $C^1$, one has the well-posedness of the three-dimensional Euler equations and for less regular initial data one has the ill-posedness. For recents surveys concurning the three-dimensional Euler equations see \cite{Bardos_Titi_2007}, \cite{Bardos_Titi_2013} and \cite{Constantin_2007}.
The situation is different for the two-dimensional Euler equations due to the work of Yudovich \cite{Yudovich_1963}.

\begin{theorem}\cite{Yudovich_1963}\label{Yudovich_2D_Euler}
Let $\omega^0 \in \LphP{\infty}$, then system \eqref{2D_Euler_vorticity} has a unique weak solution (i.e. solution in the distribution sense) $\omega \in\LtLphP{\infty}{\infty}$ corresponding to the initial data $\omega^0$ such that $\norm{\omega}_{\LtLphP{\infty}{\infty}} = \norm{\omega^0}_{\LphP{\infty}}$.
\end{theorem}

\begin{theorem}\label{Kato_theorem}\cite{Kato_1986, Kato_Ponce_1986}
Let $\bu^0\in W^{s,q}_h(\nR^2)$, with $s>1+\frac{2}{q}$, $1<q<\infty$. For any $T>0$, there exists a unique solution $\bu \in C([0,T]; W^{s,q}_h(\nR^2))\cap C^1([0,T];W^{s-1,q}_h(\nR^2))$ and $p \in C([0,T];W^{s,q}(\nR^2))$ for \eqref{2D_Euler} such that
\begin{align}\label{u_Hs_norm}
\norm{\bu(t)}_{W^{s,q}_h(\nR^2)} \leq K(t),
\end{align}
where $K(t)$ is a real-valued continuous function on $0\leq t<T$, depending on $s,q$ and $\norm{\bu^0}_{W^{s,q}_h(\nR^2)}$.
\end{theorem}

Lastly, we recall the following existence and uniqueness theorems for linear transport equations.
\begin{theorem}\cite{DiPerna_Lions_1989}\label{DiPerna_Lions_existence_1989}
Let $p\in[1,\infty]$ and $u^0\in L^p(\nR^n)$. Assume that
\begin{align*}
b \in L^1([0,T];L^1_{loc}(\nR^n)), &\qquad c \in L^1([0,T];L^1_{loc}(\nR^n)),\\
c + \nabla\cdot b \in L^1([0,T];L^q_{loc}(\nR^n)), &\qquad b \in L^1([0,T];L^q_{loc}(\nR^n)),
\end{align*}
where $\frac1p + \frac1q = 1$, and

\begin{align*}
c + \frac1p \nabla\cdot b \in L^1([0,T];L^\infty(\nR^n)), &\qquad \text{ if } p>1, \\
c, \nabla\cdot b \in L^1([0,T];L^\infty(\nR^n)), & \qquad \text{ if } p=1.
\end{align*}
If $f\in L^1([0,T];L^p(\nR^n))$, then there exists a unique weak solution $u \in L^\infty([0,T];L^p(\nR^n))$ of
\begin{align}
\pd{u}{t} + (b\cdot\nabla)u + cu = f,
\end{align}
corresponding to the initial condition $u^0$.
\end{theorem}

\begin{theorem}\cite{DiPerna_Lions_1989}\label{DiPerna_Lions_uniqueness_1989}
Let $u\in L^\infty([0,T];L^p(\nR^n))$, where $p\in[1,\infty]$,  be a solution of
\begin{align*}
\pd{u}{t} + (b\cdot\nabla) u + cu=0, \qquad u(0;x)=0.
\end{align*}
Assume that $c$, $\nabla\cdot b$ $\in $ $L^1([0,T];L^\infty(\nR^n))$, $b$ $\in$ $L^1([0,T];W^{1,q}_{loc}(\nR^n))$ where $\frac1p + \frac1q=1$ and
\begin{align*}
\frac{b}{1+\abs{x}} \in L^1([0,T];L^1(\nR^n)) + L^1([0,T];L^\infty(\nR^n)).
\end{align*}
Then $u\equiv0$.
\end{theorem}

\bigskip

\section{Global Well-posedness of Strong Solutions}\label{inviscid_well-posedness}
In this section, we aim to prove the global well-posedness of strong solutions of the inviscid system \eqref{inviscid_slow_dynamics_eq}
subject to periodic boundary conditions over any fixed arbitrary time interval $[0,T]$. We give a definition of weak solutions of system \eqref{inviscid_slow_dynamics_eq} and prove the global existence and uniqueness of such solutions. Later, we give a definition of strong solutions of system \eqref{inviscid_slow_dynamics_eq} and prove the well-posedness of such solutions.

\begin{proposition}[{\it Apriori} Estimates]\label{nu_apriori_estimates}
Assume that $\bu_h$ $\in$ $C^1([0,T]; C^\infty(\nT^2))$, $w \in C^1([0,T]; C^\infty(\nT^2))$ and $\rho\in C^1([0,T]; C^\infty(\nT^3))$ are solutions of system \eqref{inviscid_slow_dynamics_eq}
on the time interval $[0,T]$, subject to periodic boundary conditions. Then the following estimates hold true:
\begin{align}
\stackrel[0\leq t\leq T]{}{\sup}{\left(\norm{w(t)}_{\LphP{2}}^2 + \norm{\left<\rho(t)\right>_z}_{\LphP{2}}^2\right)}  &= \norm{w^0}_{\LphP{2}}^2 + \norm{\left<\rho^0\right>_z}_{\LphP{2}}^2,  \label{nu_w_rho_z_L2_estimate}\\
\stackrel[0\leq t\leq T]{}{\sup} \norm{\rho(t)}_{\LpP{2}}^2 &\leq \norm{\rho^0}_{\LpP{2}} + K_0T, \label{nu_rho_L2_estimate}
\end{align}
where $K_0$ is a constant that depends on the norms of the initial data. Moreover,
\begin{align}
&\stackrel[0\leq t\leq T]{}{\sup}\left(\norm{w(t)}_{\LphP{\infty}} + \norm{\left<\rho(t)\right>_z}_{\LphP{\infty}}\right) \leq\left(\norm{w^0}_{\LphP{\infty}} + \norm{\left<\rho^0\right>_z}_{\LphP{\infty}}\right)e^{T/Fr}, \label{nu_w_rho_z_Linfty_estimate}
\end{align}
and
\begin{align}
&\stackrel[0\leq t\leq T]{}{\sup}\left(\norm{\nabla_hw(t)}_{\LphP{2}}^2 + \norm{\nabla_h\left<\rho(t)\right>_z}_{\LphP{2}}^2\right)\notag\\
&\qquad \quad \qquad \leq \left(\norm{\nabla_hw^0}_{\LphP{2}}^2 + \norm{\nabla_h\left<\rho^0\right>_z}_{\LphP{2}}^2\right)e^{\int_0^T2\norm{\nabla_h\bu_h(s)}_{\LphP{\infty}}\,ds}.\label{nu_w_rho_z_H1_estimate}
\end{align}
\end{proposition}

\begin{proof}
Taking the $\LphP{2}$ inner product of \eqref{inviscid_slow_dynamics_eq_b} with $w$ and \eqref{viscous_slow_dynamics_eq_rho_z} with $\left<\rho\right>_z$ yield 
\begin{align*}
\frac{1}{2}\od{}{t} \norm{w}_{\LphP{2}}^2  &= -\frac{1}{Fr} \left(\left<\rho\right>_z, w \right)_h, \\
\frac{1}{2}\od{}{t} \norm{\left<\rho\right>_z}_{\LphP{2}}^2 &= \frac{1}{Fr} \left(w, \left<\rho\right>_z \right)_h.
\end{align*}
Adding the above equations implies that
\begin{align}
&\od{}{t} \left(\norm{w}_{\LphP{2}}^2 + \norm{\left<\rho\right>_z}_{\LphP{2}}^2 \right) = 0.
\end{align}
Integrating the above equation with respect to time on $[0,t]$ proves \eqref{nu_w_rho_z_L2_estimate}.

Taking the $\LpP{2}$ inner product of \eqref{inviscid_slow_dynamics_eq_c} with $\rho$ and using Young's inequality yield
\begin{align*}
\frac{1}{2}\od{}{t}\norm{\rho}_{\LpP{2}}^2 = \frac{1}{Fr}\left(w, \rho\right)
& \leq \frac{1}{Fr} \norm{w}_{\LpP{2}}\norm{\rho}_{\LpP{2}}\notag \\
& = \frac{L^{1/2}}{Fr}\norm{w}_{\LphP{2}}\norm{\rho}_{\LpP{2}}\notag \\
& \leq K_0\norm{\rho}_{\LpP{2}}, \label{nu_rho_estimate_proof}
\end{align*}
where
$$ K_0 := \frac{L^{1/2}}{Fr}\left(\norm{w^0}_{\LphP{2}}^2 + \norm{\left<\rho^0\right>_z}_{\LphP{2}}^2\right)^{1/2}. $$
Thus, we can conclude that
\begin{align*}
\od{}{t}\norm{\rho}_{\LpP{2}} \leq K_0.
\end{align*}
Integrating the above inequality with respect to time on $[0,t]$, we get that
\begin{align}
\norm{\rho(t)}_{\LpP{2}}^2 \leq \norm{\rho^0}_{\LpP{2}} + K_0t,
\end{align}
for all $t\in [0,T]$. This proves \eqref{nu_rho_L2_estimate}.

Now, we multiply \eqref{inviscid_slow_dynamics_eq_b} and \eqref{viscous_slow_dynamics_eq_rho_z} by $\left(w_n\right)^{2k-1}$ and $\left(\left<\rho_n\right>_z\right)^{2k-1}$, where $k\in \nN$, respectively, and integrate over $\nT^2$. Using H\"older inequality, we have
\begin{align*}
 \frac{1}{2k}\od{}{t}\int_{\nT^2} \left(w\right)^{2k}\, d\bx_h = -\frac{1}{Fr}\left(\left<\rho\right>_z,\left(w\right)^{2k-1}\right)_h & \leq \frac{1}{Fr}\norm{\left<\rho\right>_z}_{\LphP{2k}} \norm{w}_{\LphP{\frac{2k}{2k-1}}}\notag \\
 & = \frac{1}{Fr}\norm{\left<\rho\right>_z}_{\LphP{2k}} \norm{w}_{\LphP{2k}}^{2k-1} \end{align*}
 \begin{align*}
 \frac{1}{2k}\od{}{t}\int_{\nT^2} \left(\left<\rho\right>_z\right)^{2k}\, d\bx_h = -\frac{1}{Fr}\left(w, \left(\left<\rho\right>_z\right)^{2k-1}\right)_h & \leq \frac{1}{Fr}\norm{w}_{\LphP{2k}}\norm{\left<\rho\right>_z}_{\LphP{\frac{2k}{2k-1}}} \notag \\
  & =\frac{1}{Fr}\norm{w}_{\LphP{2k}} \norm{\left<\rho\right>_z}_{\LphP{2k}}^{2k-1}. \notag
\end{align*}
Thus,
\begin{align}
\od{}{t} \norm{w(t)}_{\LphP{2k}}& \leq \frac{1}{Fr}\norm{\left<\rho\right>_z}_{\LphP{2k}}, \quad \text{and} \quad \od{}{t} \norm{\left<\rho\right>_z}_{\LphP{2k}}& \leq \frac{1}{Fr}\norm{w}_{\LphP{2k}}.\notag
\end{align}
Adding the above equations and integrating over the time interval $[0,t]$, for $t\leq T$, imply that
\begin{align}
 \norm{w(t)}_{\LphP{2k}} + \norm{\left<\rho(t)\right>_z}_{\LphP{2k}} \leq \left(\norm{w^0}_{\LphP{2k}} + \norm{\left<\rho^0\right>_z}_{\LphP{2k}}\right) e^{t/Fr},
\end{align}
for all $t\in [0,T]$ and $k\in \nN$. Since the domain is bounded and the right-hand side bound converges, as $k\rightarrow\infty$, we can take $k\rightarrow\infty$ and obtain
\begin{align}\label{L_infty_estimate_w}
 \left(\norm{w(t)}_{\LphP{\infty}} + \norm{\left<\rho(t)\right>_z}_{\LphP{\infty}}\right) \leq \left(\norm{w^0}_{\LphP{\infty}} + \norm{\left<\rho^0\right>_z}_{\LphP{\infty}}\right) e^{t/Fr},                                                                                                                                                                                                                                                                        \end{align}
for all $t\in [0,T]$. This proves \eqref{nu_w_rho_z_Linfty_estimate}.

Now, we take the $\LphP{2}$ inner product of \eqref{inviscid_slow_dynamics_eq_b} wilth $-\Delta_hw$ and \eqref{viscous_slow_dynamics_eq_rho_z} with $-\Delta_h\left<\rho\right>_z$ and get that
\begin{align*}
&\frac{1}{2}\od{}{t} \norm{\nabla_hw}_{\LphP{2}}^2 \leq \norm{\nabla_h\bu_h}_{\LphP{\infty}}\norm{\nabla_hw}_{\LphP{2}}^2 -\frac{1}{Fr} \left(\nabla_h\left<\rho\right>_z, \nabla_hw \right)_h, \\
&\frac{1}{2}\od{}{t} \norm{\nabla_h\left<\rho\right>_z}_{\LphP{2}}^2 \leq \norm{\nabla_h\bu_h}_{\LphP{\infty}}\norm{\nabla_h\left<\rho\right>_z}_{\LphP{2}}^2 +\frac{1}{Fr} \left(\nabla_hw, \nabla_h\left<\rho\right>_z \right)_h.
\end{align*}
Adding the above equations and then integrating with respect to time on $[0,t]$ prove \eqref{nu_w_rho_z_H1_estimate}. This completes the proof.
\end{proof}

\begin{proposition}[{\it Apriori} Estimates]\label{nu_more_apriori_estimates}
Assume that $\bu_h$ $\in$ $C^1([0,T]; C^\infty(\nT^2))$, $w \in C^1([0,T]; C^\infty(\nT^2))$ and $\rho\in C^1([0,T]; C^\infty(\nT^3))$ are solutions of the system \eqref{inviscid_slow_dynamics_eq}
on the time interval $[0,T]$, subject to periodic boundary conditions. Then the following estimates hold true:
\begin{align}
&\stackrel[0\leq t\leq T]{}{\sup}\norm{\rho(t)}_{\LpP{\infty}} \leq \norm{\rho^0}_{\LpP{\infty}} + \left(\norm{w^0}_{\LphP{\infty}} + \norm{\left<\rho^0\right>_z}_{\LphP{\infty}}\right) e^{T/Fr}, \label{nu_rho_Linfty_estimate}\\
&\stackrel[0\leq t\leq T]{}{\sup}\left(\norm{\nabla_hw(t)}_{\LphP{\infty}}+ \norm{\nabla_h\left<\rho(t)\right>_z}_{\LphP{\infty}}\right) \leq {\tilde K}_0e^{\int_0^T \left(1+\norm{\nabla_h\bu_h(s)}_{\LphP{\infty}}\right)\, ds},\label{nu_w_rho_z_W1infty_estimate}
\end{align}
where ${\tilde K}_0$ is a constant that depends on the norms of the initial data.
\end{proposition}

\begin{proof}
We define
 $$\phi(t) : = \norm{\rho^0}_{\LpP{\infty}} +\left(\norm{w^0}_{\LphP{\infty}} + \norm{\left<\rho^0\right>_z}_{\LphP{\infty}} \right)e^{t/Fr}$$
 and we denote by $\Theta:= \rho - \phi(t)$. Notice that
\begin{align*}
\pd{\rho}{t} = \pd{\Theta}{t} + \od{\phi}{t} &= \pd{\Theta}{t} + \frac{1}{Fr} \left(\norm{w^0}_{\LphP{\infty}} + \norm{\left<\rho^0\right>_z}_{\LphP{\infty}}\right)e^{t/Fr}\\
& = \pd{\Theta}{t} + \frac{1}{Fr}\left(\phi(t) - \norm{\rho^0}_{\LpP{\infty}}\right).
\end{align*}
Then, $\Theta$ satisfies the evolution equation:
\begin{align}
\pd{\Theta}{t} + \frac{1}{Fr}\left(\phi(t) -\norm{\rho^0}_{\LpP{\infty}}\right) + \bu\cdot\nabla\Theta = \frac{1}{Fr}w. \label{strong_solution_Theta}
\end{align}
We can take the action of \eqref{strong_solution_Theta} with $\Theta^{+}$ and obtain
\begin{align}\label{L_2_Theta+}
&\frac{1}{2}\od{}{t} \norm{\Theta^{+}(t)}_{\LpP{2}}^2 = \frac{1}{Fr} \left(\left(w, \Theta^{+}\right) +\left(\norm{\rho^0}_{\LpP{\infty}}-\phi(t),\Theta^{+}\right)\right) \notag\\
& \qquad \leq \frac{1}{Fr} \left(\norm{w}_{L^\infty([0,t];\LphP{\infty}} + \norm{\rho^0}_{\LpP{\infty}} - \phi(t)\right)\norm{\Theta^{+}}_{\LpP{1}}.
\end{align}
Using \eqref{L_infty_estimate_w} and the definition of $\phi(t)$ we get that the right-hand side of \eqref{L_2_Theta+} is $\leq0$. Then,
\begin{align*}
 \od{\norm{\Theta^{+}}_{\LpP{2}}^2}{t} \leq 0, \quad \text{which implies that} \quad \norm{\Theta^{+}(t)}_{\LpP{2}}^2 &\leq \norm{\Theta^{+}(0)}_{\LpP{2}}^2, 
\end{align*}
for any $t\in[0,T]$.
Notice that
\begin{align*}
 \Theta(0;\bx) &= \rho(0;\bx) - \norm{\rho^0}_{\LpP{\infty}} - \frac{1}{Fr} \left( \norm{w^0}_{\LphP{\infty}} + \norm{\left<\rho^0\right>_z}_{\LphP{\infty}}\right)e^{t/Fr}
& \leq 0.
\end{align*}
Thus, $\Theta^{+}(0,\bx) = 0$ for all $\bx\in \nT^3$, which implies that $\norm{\Theta^{+}(t)}_{\LpP{\infty}} = 0$, for all $t\in [0,T]$.
That is, $\Theta^{+}(t;\bx) = 0$ a.e $\bx \in \nT^3$, for all $t\in[0,T]$, which yield that
\begin{align}
\rho(t;\bx) \leq \norm{\rho^0}_{\LpP{\infty}} + \frac{1}{Fr} \left(\norm{w^0}_{\LphP{\infty}} + \norm{\left<\rho^0\right>_z}_{\LphP{\infty}}\right) e^{t/Fr},
\end{align}
for a.e. $\bx \in\nT^3$ and all $t\in[0,T]$. This proves \eqref{nu_rho_Linfty_estimate}.

To simplify the notations in the proof, we denote by
\begin{align*}
R(t;\bx_h) &:= \left<\rho\right>_z(t;\bx_h), \\
Q^{\lambda}_w(t;\bx_h) &:= \sqrt{\abs{\nabla_hw(t;\bx_h)}^2 +\lambda}\quad \text{ and} \quad
Q^{\lambda}_R(t;\bx_h) &:= \sqrt{\abs{\nabla_hR(t;\bx_h)}^2 +\lambda},
\end{align*}
where $\lambda>0$ is any positive number. Taking the $\pd{}{x_j}$ of \eqref{inviscid_slow_dynamics_eq_b} and \eqref{viscous_slow_dynamics_eq_rho_z} yield the following system
\begin{align}
\pd{}{t} \pd{w}{x_j} + \left(\pd{\bu_h}{x_j}\cdot\nabla_h\right)w + (\bu_h\cdot\nabla_h)\pd{w}{x_j} & = -\frac{1}{Fr}\pd{R}{x_j}, \label{dxj_strong_solution_w_nu}\\
\pd{}{t} \pd{R}{x_j} + \left(\pd{\bu_h}{x_j}\cdot\nabla_h\right)R+ (\bu_h\cdot\nabla_h)\pd{R}{x_j} & = \frac{1}{Fr}\pd{w}{x_j}, \label{dxj_strong_solution_rho_z_nu}
\end{align}
for $j=1,2$. Since
\begin{align*}
\pd{Q_w^{\lambda}}{t} = \frac{\nabla_hw}{Q_w^{\lambda}} \cdot\pd{\nabla_hw}{t}, \quad \text{and} \quad \pd{Q_R^{\lambda}}{t} = \frac{\nabla_hR}{Q_R^{\lambda}} \cdot\pd{\nabla_hR}{t},
\end{align*}
we take the inner product of \eqref{dxj_strong_solution_w_nu} and \eqref{dxj_strong_solution_rho_z_nu} with $ \frac{\pd{w}{x_j}}{Q_w^{\lambda}}$ and $ \frac{\pd{R}{x_j}}{Q_R^{\lambda}}$, respectively, and then sum over $j=1,2$ and obtain
\begin{align}
\pd{Q_w^{\lambda}}{t} + \sum_{j=1}^{2} \left(\pd{\bu_h}{x_j}\cdot\nabla_h\right)w \frac{\pd{w}{x_j}}{Q_w^{\lambda}} + (\bu_h\cdot\nabla_h)Q_w^{\lambda} &=-\frac{1}{Fr} \nabla_hR \frac{\nabla_hw}{Q_w^{\lambda}}, \label{nabla_Q_w} \\
\pd{Q_R^{\lambda}}{t}+ \sum_{j=1}^{2} \left(\pd{\bu_h}{x_j}\cdot\nabla_h\right)R \frac{\pd{R}{x_j}}{Q_R^{\lambda}} + (\bu_h\cdot\nabla_h)Q_R^{\lambda} &=\frac{1}{Fr}\nabla_hw \frac{\nabla_hR}{Q_R^{\lambda}}. \label{nabla_Q_R}
\end{align}

Now, we define
\begin{align*}
\phi_\lambda(t) : = \left(\sqrt{\norm{\nabla_hw^0}_{\LphP{\infty}}+ \lambda} + \sqrt{\norm{\nabla_h\left<\rho^0\right>_z}_{\LphP{\infty}} + \lambda}\right) e^{\int_0^t\psi(s)\,ds},
\end{align*}
where $\psi(s) := \norm{\nabla_h\bu_h(s)}_{\LphP{\infty}} + \frac{1}{Fr}$. Denote by $\Theta_{w,\lambda} : = Q_w^{\lambda} - \phi_\lambda$ and by $\Theta_{R,\lambda} : = Q_R^{\lambda}- \phi_\lambda$. Its clear that $\pd{\phi_\lambda}{t}(t) = \psi(t)\phi_\lambda(t)$, thus we can replace $\pd{Q_w^{\lambda}}{t}$ in \eqref{nabla_Q_w} by $\pd{\Theta_{w,\lambda}}{t} + \psi\phi_\lambda$ and then take the inner product of the equation with $\Theta_{w,\lambda}^{+}$ and obtain
\begin{align*}
\frac{1}{2}\od{}{t} \norm{\Theta_{w,\lambda}^{+}}_{\LphP{2}}^2 + J_1 + J_2 = J_3 +J_4,
\end{align*}
where
\begin{align*}
J_1 & := \sum_{j=1}^2 \int_{\nT^2} \left(\pd{\bu_h}{x_j}\cdot\nabla_h\right)w \frac{\pd{w}{x_j}}{Q_w^{\lambda}} \Theta_{w,\lambda}^{+}\, d\bx_h, \qquad
J_2 :=  \int_{\nT^2} (\bu_h\cdot\nabla_h)\Theta_{w,\lambda}^{+}\Theta_{w,\lambda}^{+} \, d\bx_h, \\
J_3 & := -\frac{1}{Fr}\int_{\nT^2} \frac{\nabla_hR\cdot\nabla_h w}{Q_w^{\nu,\lambda}} \Theta_{w,\lambda}^{+}\, d\bx_h, \qquad\qquad
J_4  := -\phi_\lambda \psi \int_{\nT^2} \Theta_{w,\lambda}^{+}\, d\bx_h.
\end{align*}
The divergence free condition $\nabla_h\cdot\bu_h =0$ implies that $J_2=0$. By Cauchy-Schwarz inequality,
\begin{align}\label{J_2_estimate}
\abs{J_1} & \leq \int_{\nT^2} \frac{\abs{\nabla_h\bu_h}\abs{\nabla_hw}^2}{Q_w^{\lambda}}\Theta_{w,\lambda}^{+}\, d\bx_h
\leq \norm{\nabla_h\bu_h}_{\LphP{\infty}} \int_{\nT^2} \left( \Theta_{w,\lambda}+ \phi_\lambda\right)\Theta_{w,\lambda}^{+}\, d\bx_h\notag \\
& = \norm{\nabla_h\bu_h}_{\LphP{\infty}} \int_{\nT^2} \left(\Theta_{w,\lambda}^{+}\right)^2 \,d\bx_h + \norm{\nabla_h\bu_h}_{\LphP{\infty}}\phi_\lambda\int_{\nT^2} \Theta_{w,\lambda}^{+}\, d\bx_h,
\end{align}
and
\begin{align}\label{J_4_estimate}
\abs{J_3} &\leq \frac{1}{Fr}\int_{\nT^2} \frac{\abs{\nabla_hR}\abs{\nabla_hw}}{Q_w^{\lambda}} \Theta_{w,\lambda}^{+} \, d\bx_h \notag \\
& \leq\frac{1}{Fr} \int_{\nT^2} Q_R^{\lambda}\Theta_{w,\lambda}^{+}\, d\bx_h
 \leq\frac{1}{Fr} \int_{\nT^2} \Theta_{R, \lambda} \Theta_{w,\lambda}^{+} \, d\bx_h + \frac{1}{Fr}\phi_\lambda\int_{\nT^2} \Theta_{w,\lambda}^{+}\, d\bx_h \notag \\
& \leq \frac{1}{Fr}\int_{\nT^2} \Theta_{R, \lambda}^{+} \Theta_{w,\lambda}^{+} \, d\bx_h + \frac{1}{Fr}\phi_\lambda\int_{\nT^2} \Theta_{w,\lambda}^{+}\, d\bx_h \notag \\
& =\frac{1}{2Fr} \int_{\nT^2} \left( \Theta_{R,\lambda}^{+}\right)^2\, d\bx_h + \frac{1}{2Fr} \int_{\nT^2} \left(\Theta_{w,\lambda}^{+} \right)^2\, d\bx_h + \frac{1}{Fr}\phi_\lambda\int_{\nT^2} \Theta_{w,\lambda}^{+}\, d\bx_h,
\end{align}
where we used Young's inequality in the last step. Finally, from \eqref{J_2_estimate} and \eqref{J_4_estimate} we have
\begin{align}
\frac{1}{2} \od{}{t} \norm{\Theta_{w,\lambda}^{+}}_{\LphP{2}}^2 & \leq \norm{\nabla_h\bu_h}_{\LphP{\infty}}\norm{\Theta_{w,\lambda}^{+}}_{\LphP{2}}^2 + \norm{\nabla_h\bu_h}_{\LphP{\infty}} \phi_\lambda \norm{\Theta_{w,\lambda}^{+}}_{\LphP{1}} \notag \\
& \quad + \frac{1}{2Fr} \norm{\Theta_{w,\lambda}^{+}}_{\LphP{2}}^2 + \frac{1}{2Fr} \norm{\Theta_{R,\lambda}^{+}}_{\LphP{2}}^2  + \frac{1}{Fr}\phi_\lambda \norm{\Theta_{w,\lambda}^{+}}_{\LphP{1}} \notag \\
& \quad - \phi_\lambda \psi \norm{\Theta_{w,\lambda}^{+}}_{\LphP{1}}.
\end{align}
Similar argument will yield a similar inequality for $\Theta_{R,\lambda}^{+}$. After summing the two inequalities we get
\begin{align}
& \frac{1}{2}\od{}{t} \left(\norm{\Theta_{w,\lambda}^{+}}_{\LphP{2}}^2 +  \norm{\Theta_{R,\lambda}^{+}}_{\LphP{2}}^2 \right) \leq \notag \\
& \qquad \left(\norm{\nabla_h\bu_h}_{\LphP{\infty}} +\frac{1}{Fr} \right)\left(\norm{\Theta_{w,\lambda}^{+}}_{\LphP{2}}^2 +  \norm{\Theta_{R,\lambda}^{+}}_{\LphP{2}}^2 \right), \\
& \qquad + \phi_\lambda\left(\norm{\Theta_{w,\lambda}^{+}}_{\LphP{1}} +  \norm{\Theta_{R,\lambda}^{+}}_{\LphP{1}} \right)\left(\norm{\nabla_h\bu_h}_{\LphP{\infty}} + \frac{1}{Fr}-\psi  \right).
\end{align}
Since $\norm{\nabla_h\bu_h}_{\LphP{\infty}} + \frac{1}{Fr}-\psi =0$, by the definition of $\psi$, then integrating with respect to time over the interval $[0,t]$, for $t\leq T$, yields
\begin{align}
&\norm{\Theta_{w,\lambda}^{+}(t)}_{\LphP{2}}^2 +  \norm{\Theta_{R,\lambda}^{+}(t)}_{\LphP{2}}^2 \notag \\
&\qquad\qquad\leq \left(\norm{\Theta_{w,\lambda}^{+}(0)}_{\LphP{2}}^2 +  \norm{\Theta_{R,\lambda}^{+}(0)}_{\LphP{2}}^2\right) e^{\int_0^t \psi(s)\,ds} = 0,
\end{align}
since $\Theta_{w,\lambda}(0) = \sqrt{\abs{\nabla_hw^0}^2 + \lambda }- \phi_\lambda(0) $$\leq$ $0$, then  $\Theta^{+}_{w,\lambda}(0)=0$ and by a similar argument, $\Theta^{+}_{R,\lambda}(0)=0$. Therefore, for all $t\in[0,T]$ and a.e $\bx_h \in \nT^2$,
\begin{align*}
\sqrt{\abs{\nabla_hw(t)}^2 + \lambda} &\leq \phi_\lambda(t)\\ &= \left( \sqrt{\norm{\nabla_hw^0}^2_{\LphP{\infty}} + \lambda} + \sqrt{\norm{\nabla_hR^0}^2_{\LphP{\infty}} + \lambda} \right)e^{\int_0^t\psi(s)\, ds},\\
\sqrt{\abs{\nabla_hR(t)}^2 + \lambda} &\leq \phi_\lambda(t) \\ &= \left( \sqrt{\norm{\nabla_hw^0}^2_{\LphP{\infty}} + \lambda} + \sqrt{\norm{\nabla_hR^0}^2_{\LphP{\infty}} + \lambda} \right)e^{\int_0^t\psi(s)\, ds}.
\end{align*}
This implies that
\begin{align}
&\sqrt{\norm{\nabla_hw}^2_{\LphP{\infty}} + \lambda} + \sqrt{\norm{\nabla_hR}^2_{\LphP{\infty}}+ \lambda} \leq \notag \\
& \qquad 2\left( \sqrt{\norm{\nabla_hw^0}^2_{\LphP{\infty}} + \lambda} + \sqrt{\norm{\nabla_hR^0}^2_{\LphP{\infty}} + \lambda}\right)e^{\int_0^t\psi(s)\, ds},
\end{align}
for all $\lambda >0$. Since $\lambda>0$ is arbitrary, then we can take $\lambda\rightarrow0^{+}$ in the above inequality. This proves \eqref{nu_w_rho_z_W1infty_estimate} and completes the proof.
\end{proof}

\begin{definition}[Weak Solutions]
\label{weak_solution_definition_inviscid_slow_dynamics_eq}
Let $s>2$,  $\bu_h^0(\bx_h) \in \HhP{s}$, $w^0(\bx_h) \in \HhP{1} \cap\LphP{\infty}$, $ \left<\rho^0\right>_z(\bx_h) \in \HhP{1}\cap\LphP{\infty}$ and $\rho^0\in \LpP{2}$. For any $T>0$, we say that $\bu_h(t;\bx_h)$, $p(t;\bx_h)$, $w(t;\bx_h)$ and $\rho(t;\bx)$ is a weak solution of system \eqref{inviscid_slow_dynamics_eq} on the time interval $[0,T]$ if
\begin{subequations}
\begin{align}
\bu_h \in C([0,T]; \HhP{s})\cap C^1([0,T];\HhP{s-1}),\\
p \in C([0,T];\HhP{s}), \\
w \in L^\infty([0,T];  \HhP{1}\cap\LphP{\infty}), \\
\left<\rho\right>_z \in L^\infty([0,T]; \HhP{1}\cap\LphP{\infty}), \\
\rho \in L^\infty([0,T]; \LpP{2}).
\end{align}
\end{subequations}
Moreover, $\bu_h, w$ and $\rho$ satisfy \eqref{inviscid_slow_dynamics_eq} in the distribution sense; that is, for any $\phi(t;\bx_h), \chi_1(t;\bx_h), \chi_2(t;\bx_h) \in {\mathcal D}([0,T]\times\nT^2)$, with $\phi(T,\bx_h) =\chi_1(T,\bx_h) = \chi_2(T,\bx_h)= 0$, and any $\psi(t;\bx) \in {\mathcal D}([0,T]\times \nT^3)$ with $\psi(T,\bx) =0$, such that $\nabla_h\cdot\phi$ $=$ $\nabla_h\cdot\chi_1$ $=$ $\nabla_h\cdot\chi_2$ $=$ $\nabla\cdot \psi$ $=$ $0$, the following integral identities hold:
\begin{subequations}
\begin{align}
& \int_0^T \left(\bu_h(s),\phi^{'}(s)\right)_h \,ds + \int_0^T \left((\bu_h(s)\cdot \nabla_h)\phi(s), \bu_h(s)\right)_h\,ds \notag \\
& \quad\qquad \qquad \qquad \qquad\qquad\qquad\qquad\qquad\qquad\qquad\qquad= - \left(\bu_h^0, \phi^0\right)_h, \label{strong_solution_inviscid_slow_dynamics_u_h}\\
&\int_0^T \left(w(s),\chi_1^{'}(s)\right)_h \,ds  + \int_0^T \left((\bu_h(s)\cdot \nabla_h)\chi_1(s),w(s)\right)_h\,ds \notag \\
& \qquad \qquad \qquad \qquad\qquad= - \left(w^0, \chi_1^0\right)_h+ \frac{1}{Fr} \int_0^T\left(\left<\rho(s)\right>_z,\chi_1(s)\right)_h\,ds, \label{strong_solution_inviscid_slow_dynamics_w}\\
&\int_0^T \left(\left<\rho\right>_z(s),\chi_2^{'}(s)\right)_h \,ds  + \int_0^T \left((\bu_h(s)\cdot \nabla_h)\chi_2(s),\left<\rho\right>_z(s)\right)_h\,ds \notag \\
& \qquad \qquad \qquad \qquad\qquad= - \left(\left<\rho^0\right>_z, \chi_2^0\right)_h-\frac{1}{Fr} \int_0^T\left(w(s),\chi_2(s)\right)_h\,ds,
 \label{strong_solution_slow_dynamics_rho_z}\end{align}\begin{align}
&\int_0^T \left(\rho(s), {\psi^{'}}(s)\right)\,ds   +\int_0^T\left((\bu(s)\cdot \nabla)\psi(s),\rho(s)\right)\,ds \notag \\
&\qquad \qquad \qquad \qquad\qquad\;= -\left(\rho^0, \psi^0\right) -\frac{1}{Fr}\int_0^T\left(w(s),\psi(s)\right)\,ds.  \label{strong_solution_inviscid_slow_dynamics_rho}
\end{align}
\end{subequations}
\end{definition}

\begin{theorem}[Global existence and uniqueness of weak solutions]\label{weak_existence_uniqueness_inviscid}
Let $s>2$, $\bu_h^0 \in \HhP{s}$, $w^0\in \HhP{1}\cap\LphP{\infty}$, $\left<\rho^0\right>_z\in \HhP{1}\cap\LphP{\infty}$ and $\rho^0 \in \LpP{2}$. Then, for any given $T>0$, system \eqref{inviscid_slow_dynamics_eq} has a unique weak solution in the sense of Definition \ref{weak_solution_definition_inviscid_slow_dynamics_eq}. Moreover, the solution satisfies the estimates in Proposition \ref{nu_apriori_estimates}.
\end{theorem}

\begin{proof}
By Theorem \ref{Kato_theorem}, there exists a unique solution $\bu_h$ $\in$ $C([0,T];\HhP{s})$ $\cap$ $C^1([0,T]; \HhP{s-1})$ and $p$ $\in$ $C([0,T];\HhP{s})$ of the incompressible ($\nabla_h\cdot\bu_h=0$) two-dimensional Euler equations \eqref{inviscid_slow_dynamics_eq_a}. The solution is classical and will satisfy \eqref{strong_solution_inviscid_slow_dynamics_u_h} and the estimate \eqref{u_Hs_norm}. It is clear that $\nabla_h\bu_h$ $\in$ $L^\infty([0,T];\LphP{\infty})$, and by the elliptic regularity estimate \eqref{elliptic_regularity_Yudovich}, $\bu_h$ $\in$ $L^\infty([0,T];\WhP{1}{q})$ $\cap$ $L^\infty([0,T];\LphP{\infty})$, for any $q\in[2,\infty)$.

Now we consider the system:
\begin{subequations}\label{w_rho_0}
\begin{align}
\pd{w}{t} + (\bu_h\cdot\nabla_h) w  = -\frac{1}{Fr}\left<\rho\right>_z&, \\
\pd{\left<\rho\right>_z}{t}+ (\bu_h\cdot\nabla_h)\left<\rho\right>_z = \frac{1}{Fr}w&, \\
w(0;\bx_h) = w^0(\bx_h), \quad \left<\rho\right>_z(0;\bx_h)= \left<\rho\right>_z(\bx_h)&,
\end{align}
\end{subequations}

To prove the existence of $w(t;\bx_h)$ and $\left<\rho\right>_z(t;\bx_h)$,  we will follow some ideas introduced by DiPerna and Lions in \cite{DiPerna_Lions_1989}. Let $\eta(\bx_h)\in \mathcal{D}(\nR^2)$, $\int_{\nR^2}\eta(\bx_h)\,d\bx_h = 1$ . Consider $\bu_{h,\veps} = \bu_h\convh \eta_\veps$, $w^0_\veps= w^0\convh\eta_\veps$, $\left<\rho^0_\veps\right>_z = \left<\rho^0\right>_z\convh\eta_\veps$ where $\eta_\veps(.) = \frac{1}{\veps} \eta\left(\frac{.}{\veps}\right)$. Then by standard consideration, there exists a unique classical solution $w_\veps$, $\left<\rho_\veps\right>_z$ $\in$ $C([0,T];C^1_b(\nT^2))$ of
\begin{subequations}
\begin{align*}
\pd{w_\veps}{t} + (\bu_{h,\veps}\cdot\nabla_h) w_\veps  = -\frac{1}{Fr}\left<\rho_\veps\right>_z&, \\
\pd{\left<\rho_\veps\right>_z}{t}+ (\bu_{h,\veps}\cdot\nabla_h) \left<\rho_\veps\right>_z = \frac{1}{Fr}w_\veps&, \\
w_\veps(0;\bx_h) = w^0_\veps(\bx_h), \quad \left<\rho_\veps\right>_z(0; \bx_h)= \left<\rho^0_\veps\right>_z(\bx_h)&,
\end{align*}
\end{subequations}
which clearly satisfy \eqref{strong_solution_inviscid_slow_dynamics_w} and \eqref{strong_solution_slow_dynamics_rho_z}. Moreover, the solution $w_\veps$ and $\left<\rho_\veps\right>_z$ satisfy the estimates \eqref{nu_w_rho_z_Linfty_estimate} and \eqref{nu_w_rho_z_H1_estimate},
for any $\veps>0$. By the Banach-Alaoglo compactness theorem, we can extract a subsequence (which we will still denote $\{w_\veps\}_{\veps>0}$, $\{\left<\rho_\veps\right>_z\}_{\veps>0}$) such that
\begin{align}\label{w_rho_0_veps}
w_\veps \weakstar w&, \qquad \text{ in } L^\infty([0,T];\LphP{\infty}),\\
\left<\rho_\veps\right>_z \weakstar \left<\rho\right>_z &, \qquad \text{ in } L^\infty([0,T];\LphP{\infty}),
\end{align}
for some $w$, $\left<\rho\right>_z$ $\in$ $L^\infty([0,T];\LphP{\infty})$, as $\veps\strong0$. The solution $w$ and $\left<\rho\right>_z$ will inherit the estimate \eqref{nu_w_rho_z_Linfty_estimate}.
Recall that
\begin{align}\label{u_h_veps}
\bu_{h,\veps} \strong \bu_h, \qquad \text{ in } L^\infty([0,T];\LphP{1}),
\end{align}
as $\veps \strong 0$. The strong convergence \eqref{u_h_veps} and the weak-$\ast$ convergence \eqref{w_rho_0_veps} are enough to pass to the limit in \eqref{strong_solution_inviscid_slow_dynamics_w} and \eqref{strong_solution_slow_dynamics_rho_z} and show that $w$ and $\left<\rho\right>_z$ is a weak solution of system \eqref{w_rho_0}.

The uniqueness of $w$ and $\left<\rho\right>_z$ will follow by a similar argument in the proof of Theorem \ref{DiPerna_Lions_uniqueness_1989} of DiPerna and Lions since $\bu_h \in L^\infty([0,T];\WhP{1}{q}\cap\LphP{\infty})$ for any $q\in [2,\infty)$ . The argument is left to the reader. Since the solution $w$ and $\left<\rho\right>_z$ is unique, it will inherit the estimate \eqref{nu_w_rho_z_H1_estimate} using the Banach--Alaoglo compactness theorem.

The existence and the uniqueness of a solution $\rho(t;\bx)\in L^\infty([0,T];\LpP{2})$ of the linear equation \eqref{inviscid_slow_dynamics_eq_c} that satisfies \eqref{strong_solution_inviscid_slow_dynamics_rho} follows by Theorem \ref{DiPerna_Lions_existence_1989} and Theorem \ref{DiPerna_Lions_uniqueness_1989} since $\bu_h, w\in L^\infty([0,T];\HhP{1}\cap\LphP{\infty})$. Finally, we recall that the proof  of Theorem \ref{DiPerna_Lions_existence_1989} is based on the same idea of constructing an approximate sequence of solutions we used in the above proof. Thus, by using the Banach--Alaoglo compactness Theorem the solution $\rho(t;\bx)$ will inherit the estimate \eqref{nu_rho_L2_estimate}. For more details, see the proof of Theorem  \ref{DiPerna_Lions_existence_1989} in \cite{DiPerna_Lions_1989}.
\end{proof}

\begin{definition}[Strong Solutions]
\label{strong_solution_definition_inviscid_slow_dynamics_eq}
Let $s>2$,  $\bu_h^0(\bx_h) \in \HhP{s}$, $w^0(\bx_h) \in \WhP{1}{\infty}$, $ \left<\rho^0\right>_z(\bx_h) \in \WhP{1}{\infty}$ and $\rho^0(\bx)\in \LpP{\infty}$. We say that $\bu_h(t;\bx_h)$, $p(t;\bx_h)$, $w(t;\bx_h)$ and $\rho(t;\bx)$ is a strong solution  of system \eqref{inviscid_slow_dynamics_eq} on the time interval $[0,T]$ if it is a weak solution of \eqref{inviscid_slow_dynamics_eq} in the sense of Definition \ref{weak_solution_definition_inviscid_slow_dynamics_eq} and satisfies
\begin{subequations}
\begin{align}
w \in  L^\infty([0,T];  \WhP{1}{\infty}), \\
\left<\rho\right>_z \in L^\infty([0,T];\WhP{1}{\infty}), \\
\rho \in L^\infty([0,T];\LpP{\infty}).
\end{align}
\end{subequations}
\end{definition}

\begin{theorem}[Global well-posedness of strong solutions]\label{strong_wellposedness_inviscid}
Let $s>2$, $\bu_h^0 \in \HhP{s}$, $w^0\in \WhP{1}{\infty}$, $\left<\rho^0\right>_z\in \WhP{1}{\infty}$ and $\rho^0 \in \LpP{\infty}$. Then, for any given $T>0$, system \eqref{inviscid_slow_dynamics_eq} has a unique strong solution in the sense of Definition \ref{strong_solution_definition_inviscid_slow_dynamics_eq} that satisfies the estimates in Proposition \ref{nu_apriori_estimates} and Proposition \ref{nu_more_apriori_estimates}.

Assume that $\bu_h^1$, $p^1$, $w^1$, $\rho^1$ and $\bu_h^2$, $p^2$, $w^2$, $\rho^2$ are two strong solutions of system \eqref{inviscid_slow_dynamics_eq}, in the sense of Definition \ref{strong_solution_definition_inviscid_slow_dynamics_eq}, with corresponding initial data $\bu_h^{1,0}$,  $w^{1,0}$, $\rho^{1,0}$, $\bu_h^{2,0}$, $w^{2,0}$ and $\rho^{2,0}$, respectively. Define $\xi^{i} := (-\Delta)^{-1} \rho^{i}$ and $\xi^{i,0} := (-\Delta)^{-1} \rho^{i,0}$ for $i=1,2$.  Then,
\begin{align}\label{cont_dep_slow_dynamics_eq}
D(t) &\leq D(0)e^{C_0^{1,2}t}; \\
D(t)&: = \norm{(\bu_h^1-\bu_h^2)(t)}_{\LphP{2}}^2 + \norm{(w^1-w^2)(t)}_{\LphP{2}}^2 \notag\\
&\quad + \norm{\left<\rho^1-\rho^2\right>_z(t)}_{\LphP{2}}^2+\norm{\nabla(\xi^1-\xi^2)(t)}_{\LpP{2}}^2, \notag
\end{align}
for all $t\in[0,T]$, where $C_{0}^{1,2} = C_0^{1,2}(L,T, Fr)$ is a constant that depends on $T$, $L$, $Fr$ and may depend on the norms of the initial data $\bu_h^{1,0}$, $\bu_h^{2,0}$, $w^{1,0}$, $w^{2,0}$, $\rho^{1,0}$ and  $\rho^{2,0}$.
\end{theorem}

\begin{proof}
The existence and uniqueness of a weak solution follows by Theorem \ref{weak_existence_uniqueness_inviscid}. The solution will satisfy the estimates in Proposition \ref{nu_more_apriori_estimates} by the same argument presented in the proof of Theorem \ref{weak_existence_uniqueness_inviscid}. This proves that the solution is a strong solution.
Assume that $\bu_h^1$, $p^1$, $w^1$, $\rho^1$ and $\bu_h^2$, $p^2$, $w^2$, $\rho^2$ are two strong solutions of system \eqref{inviscid_slow_dynamics_eq} with corresponding initial data $\bu_h^{1,0}$, $w^{1,0}$, $\rho^{1,0}$ , $\bu_h^{2,0}$, $w^{2,0}$ and $\rho^{2,0}$, respectively. Following the idea introduced in \cite{Larios_Lunasin_Titi_2010}, we introduce the stream function $\xi^i$, where $\rho^i := -\Delta \xi^i$ and $\int_{\nT^3} \xi^i(t;\bx)\, d\bx = 0$, for all $t\in [0,T]$ with corresponding initial condition $\xi^{i,0}$, for $i=1,2$. We denote by ${\tilde \bu}_h := \bu_h^1-\bu_h^2$, ${\tilde p}:= p^1-p^2$, ${\tilde w} := w^1-w^2$, ${\tilde \rho} := \rho^1-\rho^2$ and ${\tilde \xi} := \xi^1 -\xi^2$. It is easy to check that ${\tilde \bu}_h$, ${\tilde w}$, ${\tilde \rho}$ and ${\tilde \xi}$ will satisfy the functional equations
\begin{subequations}\label{cont_dep_slow_limiting_eq}
\begin{align}
&\pd{{\tilde \bu}_h}{t} + (\bu_h^1\cdot\nabla_h){\tilde \bu}_h + ({\tilde \bu}_h\cdot\nabla_h)\bu_h^2  +\nabla_h{\tilde p}= 0,  \;\; \quad \text{ in } C([0,T];\HhP{s-1}), \label{cont_dep_slow_limiting_eq_u}\\
&\pd{{\tilde w}}{t} + (\bu_h^1\cdot\nabla_h){\tilde w} + ({\tilde \bu}_h \cdot \nabla_h) w^2 = -\frac{1}{Fr} \left<{\tilde \rho}\right>_z, \; \;\quad \text{ in }  L^2([0,T]; \LphP{2}), \label{cont_dep_slow_limiting_eq_w}\\
&\pd{\left<{\tilde \rho}\right>_z}{t} + (\bu_h^1\cdot\nabla_h)\left<{\tilde \rho}\right>_z + ({\tilde \bu}_h\cdot \nabla_h)\left<\rho^2\right>_z  = \frac{1}{Fr} {\tilde w}, \;  \text { in } L^2([0,T];\LphP{2}), \label{cont_dep_slow_limiting_eq_rho_z}\\
&-\pd{\Delta {\tilde \xi}}{t} - (\bu_h^1\cdot\nabla_h) \Delta{\tilde \xi} - ({\tilde \bu}_h\cdot \nabla_h)\Delta \xi^2 - w^1\pd{\Delta{\tilde \xi}}{z} + {\tilde w}\pd{\rho^2}{z}\notag \\ &\quad \qquad\qquad\qquad\qquad\qquad\qquad\qquad\qquad=\frac{1}{Fr} {\tilde w}, \; \text{ in } L^2([0,T];\HP{-1}). \label{cont_dep_slow_limiting_eq_rho}
\end{align}
\end{subequations}
Clearly, we can take the inner product of \eqref{cont_dep_slow_limiting_eq_u} with ${\tilde \bu}_h$, \eqref{cont_dep_slow_limiting_eq_w} with ${\tilde w}$, \eqref{cont_dep_slow_limiting_eq_rho_z} with $\left<{\tilde \rho}\right>_z$ and \eqref{cont_dep_slow_limiting_eq_rho} with ${\tilde \xi}$. Using the divergence free condition $\nabla_h\cdot \bu_h^1$ $=$ $\nabla_h\cdot\bu_h^2$ $=$ $\nabla_h\cdot{\tilde \bu}_h$ $=$ $0$, $\pd{w^1}{z}$ $=$ $\pd{w^2}{z}$ $=$ $\pd{{\tilde w}}{z}$ $=$ $0$, integration by parts, H\"older inequality and Young's inequality we can show that
\begin{align}
&\frac{1}{2} \od{}{t}\norm{{\tilde \bu}_h}_{\LphP{2}}^2 \leq \abs{\left(({\tilde\bu}_h\cdot\nabla_h){\tilde \bu}_h^2, {\tilde\bu}_h\right)_h}\leq \norm{\nabla_h\bu_h^2}_{\LphP{\infty}} \norm{{\tilde \bu}_h}_{\LphP{2}}^2,\label{tilde_u_slow_limiting_eq} \\
&\frac{1}{2} \od{}{t} \left(\norm{{\tilde w}}_{\LphP{2}}^2 + \norm{\left<{\tilde \rho}\right>_z}_{\LphP{2}}^2 \right)\notag\\
&\qquad \leq \abs{\left(({\tilde\bu}_h\cdot\nabla_h)w^2, {\tilde w}\right)_h} + \abs{\left(({\tilde\bu}_h\cdot\nabla_h)\left<\rho^2\right>_z, \left<{\tilde \rho}\right>_z\right)_h} \notag \\
&\qquad\leq \norm{{\tilde \bu}_h}_{\LphP{2}}\left(\norm{\nabla_h w^2}_{\LphP{\infty}}\norm{{\tilde w}}_{\LphP{2}}+ \norm{\nabla_h\left< \rho^2\right>_z}_{\LphP{\infty}} \norm{\left<{\tilde \rho}\right>_z}_{\LphP{2}}\right)\notag \\
&\qquad\leq \left(\norm{\nabla_h w^2}_{\LphP{\infty}}+\norm{\nabla_h\left<\rho^2\right>_z}_{\LphP{\infty}}\right)D(t), \label{tilde_w_rho_z_slow_limiting_dynamics_eq}
\end{align}
and
\begin{align}
\frac{1}{2} \od{}{t}\norm{\nabla{\tilde \xi}}_{\LpP{2}}^2
&\leq \abs{\left((\bu_h^1\cdot\nabla_h) {\tilde \xi},\Delta{\tilde\xi}\right)} + \abs{\left(({\tilde \bu}_h\cdot \nabla_h) {\tilde\xi}, \Delta\xi^2\right)} + \abs{\left(w^1\pd{{\tilde \xi}}{z}, \Delta{\tilde\xi}\right)} \notag\\
& \quad+ \abs{\left({\tilde w}\pd{{\tilde\xi}}{z}, \rho^2\right)} +\frac{1}{Fr}\abs{\left({\tilde w},{\tilde\xi}\right)}\notag\\
&  \leq\sum_{j=1}^2\abs{\left(\pd{\bu_h^1}{x_j}, \nabla_h{\tilde \xi}\pd{{\tilde \xi}}{x_j}\right)}+ \abs{\left(({\tilde \bu}_h\cdot \nabla_h) {\tilde\xi}, \Delta\xi^2\right)}+ \abs{\left(\nabla_hw^1\pd{{\tilde\xi}}{z},\nabla_h{\tilde\xi}\right)}\notag\\
& \quad +\abs{\left({\tilde w}\pd{{\tilde\xi}}{z}, \rho^2\right)}+\frac{1}{Fr}\abs{\left({\tilde w},{\tilde\xi}\right)}\notag
\end{align}
\begin{align}
&\leq\norm{\nabla_h{\bu}_h^1}_{\LpP{\infty}}\norm{\nabla_h{\tilde\xi}}_{\LpP{2}}^2+ \norm{{\tilde\bu}_h}_{\LpP{2}}\norm{\nabla_h{\tilde\xi}}_{\LpP{2}}\norm{\Delta\xi^2}_{\LpP{\infty}}\notag\\
&\quad+\norm{\nabla_hw^1}_{\LpP{\infty}}\norm{\nabla{\tilde\xi}}_{\LpP{2}}^2+\norm{\rho^2}_{\LpP{\infty}}\norm{{\tilde w}}_{\LpP{2}}\norm{\nabla{\tilde\xi}}_{\LpP{2}} \notag\\
&\quad+ \frac{1}{Fr}\norm{{\tilde w}}_{\LpP{2}}\norm{{\tilde\xi}}_{\LpP{2}}. \notag
\end{align}
Using Poincar\'e inequality, Lemma \ref{Poincare}, we may conclude that
\begin{align}\label{tilde_xi_slow_limiting_dynamics_eq}
&\frac12\od{}{t}\norm{\nabla{\tilde\xi}}_{\LpP{2}}^2 \leq C(L, Fr)\left(\norm{\nabla_h\bu_h^1}_{\LphP{\infty}}+ \norm{\rho^2}_{\LpP{\infty}}+ \norm{\nabla_hw^1}_{\LphP{\infty}}\right)D(t),
\end{align}
where $C(L,Fr)$ is a constant that depends on the size of the domain $L$ and the Froude number $Fr$. Recall that by assumption $\bu_h^1$, $\rho^2$ and $w^1$ are strong solutions in the sense of Definition \ref{strong_solution_definition_inviscid_slow_dynamics_eq}, and they satisfy the estimates in Proposition \ref{nu_apriori_estimates} and Proposition \ref{nu_more_apriori_estimates} with corresponding initial data $\bu_h^{1,0}$, $\rho^{2,0}$ and $w^{1,0}$, respectively.
Thus, adding \eqref{tilde_u_slow_limiting_eq}, \eqref{tilde_w_rho_z_slow_limiting_dynamics_eq} and \eqref
{tilde_xi_slow_limiting_dynamics_eq} imply that
\begin{align}
\od{}{t} D{t} \leq C_{0}^{1,2}(T,L, Fr) D(t),
\end{align}
where $C_{0}^{1,2}(T,L, Fr)$ is a constant that depends on $T$, $L$ ,$Fr$ and the norms of the initial data $\bu_h^{1,0}$, $w^{1,0}$ and $\rho^{2,0}$. Integrating the above inequality on the time interval $[0,t]$, for $t\leq T$, proves \eqref{cont_dep_slow_dynamics_eq} and completes the proof.
\end{proof}
\bigskip

\section{Global existence and uniqueness of weak solutions using vorticiy formulation}\label{inviscid_slow_dynamics_eq_vorticity}
In this section, we aim to prove the global existence and uniqueness of weak solutions of the inviscid {\it slow-limiting dynamics model} in vorticity formulation. We define the vorticity $\omega = \nabla_h\times\bu_h$. As in the case of the two-dimensional incompressible Euler equations in vorticity formulation, we have the analogue  of the two-dimensional periodic Biot-Savart law \eqref{conv}.
We explicitly restrict ourselves to the unique solution $\bu_h$ of the elliptic system: $\nabla_h\times\bu_h=\omega$ and  $\nabla_h\cdot \bu_h =0$, that satisfies the side condition $\int_{\nT^2} \bu_h(x,y) \,dx\,dy = 0$.

Since we are considering periodic boundary conditions, we can write
\begin{align*}
 \rho(\bx) = \sum\limits_{k=-\infty}^{\infty} \rho_k(\bx_h)e^{\frac{2\pi}{L}ikz},\qquad \text{where} \quad  \rho_k(\bx_h) = \frac{1}{L}\int_0^L \rho(\bx)e^{-\frac{2\pi}{L}ikz}\, dz,
 \end{align*}
are periodic in $\nT^2$ for each $k$. Notice that $\rho_0(\bx_h) \equiv \left<\rho\right>_z(\bx_h)$. We may take the horizontal curl of equation \eqref{inviscid_slow_dynamics_eq_a} and take the Fourier transform of equation \eqref{inviscid_slow_dynamics_eq_c} and rewrite system \eqref{inviscid_slow_dynamics_eq} in vorticity-Fourier transform formulation as 
 \begin{subequations}\label{inviscid_slow_dynamics_eq_vorticity_k}
\begin{align}
\frac{\partial \omega}{\partial t} + (\bu_h\cdot\nabla_h)\omega &=0, \label{inviscid_slow_dynamics_eq_voriticity_k_a} \\
\frac{\partial w}{\partial t} + (\bu_h\cdot\nabla_h)w &=   - \frac{1}{Fr}\left<\rho\right>_z, \label{inviscid_slow_dynamics_eq_voriticity_k_b}\\
\pd{\left<\rho\right>_z}{t} + (\bu_h\cdot\nabla_h)\left<\rho\right>_z & = \frac{1}{Fr}w, \\
\frac{\partial \rho_k}{\partial t} + (\bu_h\cdot\nabla_h)\rho_k + ikw\rho_k&= 0 , \label{inviscid_slow_dynamics_eq_voriticity_k_c}\\
\nabla_h\cdot \bu_h = 0, \qquad \nabla\cdot\bu = 0, \qquad \omega &= \nabla_h\times\bu_h,\label{inviscid_slow_dynamics_eq_voriticity_k_d}\\
\omega(0;\bx_h) =\omega^0(\bx_h), \quad w(0;\bx_h)=w^0(\bx_h), \quad \left<\rho\right>_z(0;\bx_h) &= \left<\rho^0\right>_z(\bx_h),
\end{align}
and
\begin{align}
 \rho_k(0;\bx_h)&=\rho_k^0(\bx_h),\qquad \text{where} \quad \rho_k^0(\bx_h) = \frac{1}{L}\int_0^L \rho^0(\bx)e^{-\frac{2\pi}{L}ikz}\, dz,
\end{align}
\end{subequations}
for each $k\in \nZ \backslash \{0\}$.

\begin{definition}[Weak Solutions]
\label{weak_solution_definition_inviscid_slow_dynamics_eq_vorticity_k}
For any $1< q\leq \infty$, let $\omega^0(\bx_h) \in \LphP{\infty}$, $w^0(\bx_h) \in \LphP{\infty}$, $ \left<\rho^0\right>_z(\bx_h) \in \LphP{\infty}$ and $\rho^0_k(\bx_h)\in \LphP{q}$, for each $k\in \nZ \backslash \{0\}$. For any $T>0$, we say that $\omega(t;\bx_h)$, $w(t;\bx_h)$ and $\rho(t;\bx)= \sum\limits_{k=-\infty}^{\infty} \rho_k(\bx_h)e^{\frac{2\pi}{L}ikz}$ is a weak solution  of system \eqref{inviscid_slow_dynamics_eq_vorticity_k} on the time interval $[0,T]$ if
\begin{subequations}
\begin{align}
\omega \in L^\infty([0,T[;\LphP{\infty}),\\
w \in L^\infty([0,T];  \LphP{\infty}), \\
\left<\rho\right>_z \in  L^\infty([0,T];\LphP{\infty}), \\
\rho_k \in L^\infty([0,T]; \LphP{q}),
\end {align}
\end{subequations}
for each $k\in \nZ \backslash \{0\}$. Moreover, $\omega, w$ and $\rho_k$ satisfy \eqref{inviscid_slow_dynamics_eq_vorticity_k} in the distribution sense for each $k$; that is, for any $\phi(t;\bx_h), \chi_1(t;\bx_h), \chi_2(t;\bx_h)  \in {\mathcal D}([0,T]\times\nT^2)$, with $\phi(T,\bx_h) =\chi_1(T,\bx_h)= \chi_2(T,\bx_h)=0$, and any $\psi(t;\bx_h) \in {\mathcal D}([0,T]\times \nT^2)$ with $\psi(T,\bx_h) =0$, such that $\nabla_h\cdot\phi$ $=$ $\nabla_h\cdot\chi_1$ $=$ $\nabla_h\cdot\chi_2$ $=$  $\nabla_h\cdot\psi$ $=0$, the integral identities \eqref{strong_solution_inviscid_slow_dynamics_w}, \eqref{strong_solution_slow_dynamics_rho_z} and
\begin{subequations}
\begin{align}
& \int_0^T \left(\omega(s),\phi^{'}(s)\right)_h \,ds + \int_0^T \left((\bu_h(s)\cdot \nabla_h)\phi(s), \omega(s)\right)_h\,ds = - \left(\bu_h^0, \phi^0\right)_h,\label{weak_solution_inviscid_slow_dynamics_omega}\\
&\int_0^T \left(\rho_k(s), {\psi^{'}}(s)\right)_h\,ds   +\int_0^T\left((\bu_h(s)\cdot \nabla_h)\psi(s),\rho_k(s)\right)_h\,ds\notag \\ 
&\qquad \qquad \qquad \qquad\qquad= + ik\int_0^T\left(w(s)\rho_k(s), \psi(s)\right)_h\,ds-\left(\rho^0_k, \psi^0\right)_h,  \label{weak_solution_inviscid_slow_dynamics_rho}
\end{align}
\end{subequations}
hold for each $k\in \nZ  \backslash \{0\}$.
\end{definition}

\begin{theorem}[Global existence and uniqueness of weak solutions]\label{weak_existence_uniqueness_inviscid_vorticity}
For any  $1< q\leq \infty$, let $\omega^0 \in \LphP{\infty}$, $w^0\in\LphP{\infty}$, $\left<\rho^0\right>_z\in\LphP{\infty}$ and $\rho^0_k \in \LphP{q}$, for each $k\in \nZ  \backslash \{0\}$. Then, for any given $T>0$, system \eqref{inviscid_slow_dynamics_eq_vorticity_k} has a unique weak solution in the sense of Definition \ref{weak_solution_definition_inviscid_slow_dynamics_eq_vorticity_k}.
\end{theorem}

\begin{proof}
Given $\omega^0(\bx_h)\in\LphP{\infty}$ and $T>0$, by Theorem \ref{Yudovich_2D_Euler}, the two-dimensional incompressible Euler equations in vorticity formulation have a unique solution $\omega(t;\bx_h) \in L^\infty([0,T]; \LphP{\infty})$ such that
\begin{align}
\norm{\omega(t)}_{\LphP{p}} = \norm{\omega^0}_{\LphP{p}},
\end{align}
for any $p\in[1,\infty]$.
 Moreover, by the elliptic regularity estimate \eqref{elliptic_regularity_Yudovich}, $\bu_h(t;\bx_h)$ $\in$ $L^\infty([0,T];\WhP{1}{p})$ for any $p\in[1,\infty)$ and
\begin{align}
\norm{\bu_h(t)}_{\WhP{1}{p}} \leq C p \norm{\omega^0}_{\LphP{\infty}},
\end{align}
for all $t\in[0,T]$.  The existence and uniqueness of a weak solution $w(t;\bx_h)\in L^\infty([0,T];\LphP{\infty})$ and $\left<\rho\right>_z(t;\bx_h) \in L^\infty([0,T];\LphP{\infty})$ that satisfy the inviscid system \eqref{w_rho_0}  in the weak sense and the estimate \eqref{nu_w_rho_z_Linfty_estimate} is presented in the proof of Theorem \ref{weak_existence_uniqueness_inviscid}.

Since $\bu_h, w \in L^\infty([0,T];\LphP{\infty})$ and $\nabla_h\cdot\bu_h=0$, then the existence of a solution $\rho_k(t;\bx_h)\in L^\infty([0,T];\LphP{q})$ for $k\in \nZ\backslash \{0\}$ of \eqref{inviscid_slow_dynamics_eq_voriticity_k_c} will follow directly by Theorem \ref{DiPerna_Lions_existence_1989}. The uniqueness of the solution $\rho_k(t;\bx_h)$ for each $k\in \nZ \backslash \{0\}$ follow by Theorem \ref{DiPerna_Lions_uniqueness_1989}  since $\bu_h\in L^\infty([0,T];\WhP{1}{p})$ for any $p\in[1,\infty)$. This completes the proof. 
\end{proof}
\bigskip

\section*{Acknowledgements}
This paper is dedicated to Professor Marshall Slemrod on the occasion of his $70^{th}$ birthday as a token of friendship and admiration. The work of C.C. is supported in part by the NSF grant DMS-1109022. The work of  E.S.T.  is supported in part by the NSF grants  DMS-1009950, DMS-1109640 and DMS-1109645, and by the Minerva Stiftung/Foundation.

\bigskip


\end{document}